\documentclass[a4paper,11pt]{article}

\usepackage{graphicx}
\usepackage[affil-it]{authblk}

\usepackage{alltt} 

\usepackage{amsfonts}
\usepackage{amsmath}
\usepackage{amssymb}
\usepackage{amsthm}

\newtheorem{theorem}{Theorem}[section]
\newtheorem{proposition}[theorem]{Proposition}
\newtheorem{lemma}[theorem]{Lemma}
\newtheorem{corollary}[theorem]{Corollary}

\theoremstyle{definition}
\newtheorem{definition}[theorem]{Definition}
\theoremstyle{remark}
\newtheorem{remark}[theorem]{Remark}
\theoremstyle{remark}

\theoremstyle{remark}

\newcommand{\G}{G}

\begin{document}

\title{Presentations of Topological Full Groups by Generators and Relations}

\author{Rostislav Grigorchuk%
  \thanks{The research of the first author was supported by NSF grant DMS-1207669.}}
\affil{Department of Mathematics, Texas A\&M, College Station, TX \\ grigorch@math.tamu.edu}

\author{Konstantin Medynets%
  \thanks{The second author was supported by NSA  grant H98230CCC5334.}}
\affil{Department of Mathematics, U.S. Naval Academy, Annapolis, MD  \\ medynets@usna.edu}

\date{}
\maketitle

\begin{abstract} We describe generators and defining relations for the  commutator subgroup of topological full groups of minimal subshifts. We show that the  word problem  in a topological full group is solvable if and only if the language of the underlying subshift is recursive.
\end{abstract}

\section{Introduction}  Topological full groups  (abbreviated as  TFGs) first appeared in the theory of crossed-product $C^*$-algebras. A TFG can be defined as the group of automorphisms of a crossed product $C^*$-algebra preserving the maximal Abelian subalgebra  modulo its center \cite[Lemma 5.1 and Theorem 5.2]{Putnam:1989}, see also \cite[Theorem 1]{Tomiyama:1996}. The TFGs  were proven to be  complete invariants of the restricted isomorphism class of the crossed-product $C^*$-algebra \cite[Theorem 2]{Tomiyama:1996}.

The topological full groups also play a major role in the classification theory of symbolic dynamical systems. It turns out that two minimal dynamical systems are flip conjugate (recall that two dynamical systems are called {\it flip conjugate} if they are conjugate or if one is conjugate to the inverse of the other) if and only if the associated TFGs  are isomorphic as abstract groups  \cite{GirdanoPutnamSkau:1999}, \cite{BoyleTomiyama}. We note that the term ``topological'' in TFG has been historically used to refer  to the fact that these groups are associated with topological dynamical systems. No group topology is assumed on TFGs.

Recently, the construction of topological full groups and the interplay between their algebraic properties and the dynamical properties of underlying symbolic systems were used to   establish the existence of {\it infinite  finitely generated simple amenable  groups}. It turns out that the commutator subgroups of TFGs associated with minimal subshifts over finite alphabets have the desired characteristics \cite{JuschenkoMonod}.
 The  discussion of  algebraic properties of full groups can be found in    \cite{BezuglyiMedynets:2008}, \cite{GirdanoPutnamSkau:1999},  \cite{GrigorchukMedynets},  \cite{Matui:2006, Matui:2013}, and \cite{Medynets:2011}.

The goal of the  paper is to find a presentation of topological full groups and relate it to properties of the underlying dynamical system.  Theorem \ref{TheoremIntroMain} is the main result of  the paper.   It describes the set of defining relations and identifies groups with solvable word problem. The result is established in    Theorem \ref{TheoremMain}, Theorem \ref{TheoremTietzeTransformations}, and Theorem \ref{TheoremWordProblem}.

 Let $(\Omega,T)$ be a minimal subshift over a finite alphabet. Denote by $G_T$ the topological full group of $(\Omega,T)$ and $G_T'$ the commutator subgroup of $G_T$ (Definition \ref{DefinitionFullGroup}).  Denote by $L_n(\Omega)$ the set of words of length $n$ appearing in sequences of $\Omega$, $n\geq 1$. The language of the subshift is defined as $L(\Omega) = \bigcup_{n\geq 1} L_n(\Omega)$. The base of the topology on $\Omega$ comprises cylinder sets  $(v,i) = \{\omega\in \Omega : \omega_{-i} = v_0,\ldots,\omega_{|v|-i} = v_{|v|-1}\}$, where $v\in L(\Omega)$, $i\in \mathbb Z$, and $|v|$ is the length of the word $v$. By a {\it cylinder partition} of $(v,i)$, we mean a partition into cylinder sets, see Definition \ref{DefinitionDisjoint}. In the following result, we use symbols $x_{(v,i)}$, $(v,i)\in L(\Omega)\times \mathbb Z$, as a base for the free group.

 \begin{theorem}\label{TheoremIntroMain}   Let $(\Omega,T)$ be a minimal subshift over a finite alphabet.
 (1) There exists $n\geq 3$ such that the commutator subgroup of the topological full group $G_T'$ is isomorphic to the group $\Gamma_\Omega$ generated by  $$<x_{(w,k)}, w\in L(\Omega),\; |w|\geq n,\; k\in \mathbb Z>,$$  subject to the following relations: for every $w,v\in L(\Omega)$, $|w|,|v|\geq n$, $i,j\in \mathbb Z$,  and a cylinder partition $C$ of $(w,i)$,
\begin{flalign}
& \left(x_{(w,i)}\right)^3=1  \label{RelIntro1} \\
& \left(x_{(w,i)}\cdot x_{(w,i+1)}\right)^2 = 1    \\
& x_{(w,i+1)} = x_{(w,i+2)} x_{(w,i)}^{-1} x_{(w,i+2)}^{-1} x_{(w,i)} \\
& x_{(w,i)} = \prod_{(s,k) \in C}x_{(s,k)}
\end{flalign}
and
\begin{equation} \label{RelIntro5}  [x_{(w,i)},x_{(v,j)}] = 1,
\end{equation}
whenever the cylinder sets $(w,i)$, $(w,i+1)$, $(w,i+2)$, $(v,j)$, $(v,j+1)$, $(v,j+2)$ are mutually disjoint.

Furthermore, the group $\Gamma_\Omega$ is generated by the elements $<x_{(w,1)}$, $w\in L_n(\Omega)>$.

(2) The set of relations (\ref{RelIntro1})--(\ref{RelIntro5}) is recursive if and only if  the language $L(\Omega)$ is recursive. If $L(\Omega)$ is recursive, then $n$ is effectively computable.

(3) The group $G_T'$ has decidable word problem if and only if $L(\Omega)$ is recursive.
 \end{theorem}

In the proof of Theorem \ref{TheoremMain} we show that each minimal subshift $(\Omega,T)$ over a finite alphabet can be topologically  conjugated to a minimal subshift for which $n = 3$ (see the statement of Theorem \ref{TheoremIntroMain}). The presentation of $G_T'$, in the case $n=3$, in terms of the finite generating set $\{x_{(w,1)} : w\in L_3(\Omega)\}$ is given in Theorem \ref{TheoremTietzeTransformations}. We note that the fact that  the commutator subgroup $G_T'$ is finitely-generated was originally established by Matui in \cite{Matui:2006}. He also showed that the commutator subgroup could not be finitely presented, see also \cite{GrigorchukMedynets} for an alternative proof.

One way to better understand the  relations in Theorem \ref{TheoremIntroMain} is to compare them to those defining the alternating groups (Lemma \ref{LemmaAlternaingGroupRelations}). The proof of the theorem relies on the fact  that the commutator subgroup of $G_T$ can be written as $G_T' = G_1\cdot G_2$, where $G_1$ and $G_2$ are certain subgroups of $G_T'$ isomorphic to increasing unions of products of alternating groups. In the proof of Theorem \ref{TheoremMain} we show that combining the relations for $G_1$ and $G_2$ we obtain a complete set of relations needed to describe $G_T'$.

According to Theorem \ref{TheoremIntroMain}, to construct a TFG with decidable word problem one needs to take a system with recursive language $L(\Omega)$. Such systems can be found amongst  substitution, Toeplitz, and Sturmian systems (see definitions in \cite{DurandHostSkau}, \cite{Downarowicz:2005}, and \cite{Allouce:AutomaticSequences}, respectively).  The following result gives a dynamical reformulation of the isomorphism problem. The proof can be found in \cite{BezuglyiMedynets:2008}, see also \cite{GirdanoPutnamSkau:1999}.

\begin{theorem} Let $(\Omega_1,T_1)$ and $(\Omega_2,T_2)$ be minimal subshifts. Then $G_{T_1}'$ and $G_{T_2}'$ are isomorphic as abstract groups if and only if the dynamical systems are flip-conjugate, i.e., there exists a  homeomorphism $f : \Omega_1\rightarrow \Omega_2$ such that $T_2 = f\circ T_1\circ f^{-1}$ or $T_2 = f\circ T_1^{-1}\circ f^{-1}$.

In particular, the decidability of  isomorphism problem for TFGs is equivalent to the decidability of  flip-conjugacy for minimal subshifts.
\end{theorem}

 We note that the flip-conjugacy is decidable, for example, in the class of constant length primitive substitutions  \cite{CovenQuasYassawi:2015}.  Fabien~Durand recently announced that they solved the decidability problem for the class of all primitive substitution systems [private communication].

The structure of the paper is the following. In Section \ref{SectionPreliminaries}, we introduce notations, main definitions, and state the necessary results from the theory of dynamical systems.
  In Section \ref{SectionGenerators} we describe a particular set of generators for commutator subgroups of  topological full groups and establish several relations for these generators  (Corollary \ref{CorollarySigmaRelations}).    Section \ref{SectionRelations} is devoted to the proof of Theorem \ref{TheoremIntroMain}.

%
%
%

\section{Preliminaries}\label{SectionPreliminaries}

In this section we introduce  main notations and give necessary definitions from the theory of topological dynamical systems.

\subsection{Subshifts}

Fix  a finite alphabet  $A$. Let  $T : A^\mathbb Z \rightarrow A^\mathbb Z$ be the {\it left  shift} on the space of two-sided sequences over $A$, i.e., $T(\omega)_n = \omega_{n+1}$ for all $n\in  \mathbb Z$  for every $\omega\in A^\mathbb Z$. Let  $\Omega$ be a closed $T$-invariant subset  of $A^\mathbb Z$. The pair $(\Omega, T)$ is called a {\it subshift over a finite alphabet}.  The subshift is called {\it minimal} if every $T$-orbit is dense in $\Omega$. Throughout this paper,  $(\Omega,T)$ will stand for a minimal subshift over a finite alphabet.

\begin{definition}\label{DefinitionFullGroup} Let $(\Omega, T)$ be a  minimal subshift. Denote by $G_T$ the set of homeomorphisms $S :\Omega \rightarrow \Omega$ such that $S(x) = T^{f_S(x)}(x)$ for every $x\in X$, where $f_S: \Omega\rightarrow \mathbb Z$ is a continuous function. We will refer  to $f_S$ as the {\it orbit cocycle of $S$}.  The cocycle identity $f_{S_1S_2}(x) = f_{S_1}(S_2(x))+ f_{S_2}(x)$ shows that the family $G_T$ is a group, termed the {\it topological full group} of the system $(\Omega,T)$. It will be  abbreviated  {\it TFG}.
When $T$ is a minimal subshift over a finite alphabet, the group $G_T$ is countable, its  commutator subgroup $G_T'$ is simple and finitely-generated  \cite{BezuglyiMedynets:2008}  and \cite{Matui:2006}.
All groups appearing in the paper will be assumed to be acting on the left on $\Omega$.
\end{definition}

Denote by  $A^*$ the set of all non-empty words over the alphabet $A$. For a subshift $(\Omega,T)$, denote by  $L_n(\Omega)$ the set of words of length $n$ appearing in sequences of $\Omega$. The set  $L(\Omega) = \bigcup_{n\geq 1}L_n(X)$ is called the {\it language} of the subshift $\Omega$. The length of each word $u\in A^*$ will be denoted by $|u|$.

For any pair of words $u,v\in L(\Omega)$ such that $uw\in L(\Omega)$, denote by $[u.v]$ the set of sequences $\omega\in \Omega$ such that $\omega_{-i} = u_{|u|-i}$ for $i=1,\ldots,|u|$ and $\omega_i = v_i$ for $i=0,\ldots,|v|-1$.  We note that the cylinder sets $[u.v]$, $u,v\in L(\Omega)$, generate the topology of $\Omega$. It will be sometimes convenient to use  $[u|v]$ for $[u.v]$. For $v\in L(\Omega)$ and $i\in \mathbb Z$, set $(v,i) = T^{i}[.v]$, as defined in the introduction.

\medskip \noindent  {\bf Standing Assumption.} {\it Throughout the paper we will assume that no word of length five in $L(\Omega)$ has repeated letters,} i.e. if $w = w_0 \cdots w_4\in L_5(\Omega)$, then $w_i\neq w_j$ whenever $i\neq j$. The systems satisfying this assumption will be referred to as {\it satisfying the condition ($\dag$)}.

 This assumption will lead to a  simpler  description of defining relations in $G_T$, Theorem \ref{TheoremMain}.  The following proposition  shows that any minimal subshift can be made to  satisfy the condition ($\dag$). Furthermore, the transition to the conjugate system will not affect the recursiveness of the language, which will be important in solving word problem  in full groups, Theorem \ref{TheoremWordProblem}.

\begin{definition}
(i) Let $S$ be a finite set. A set $R\subset S^*$ is called {\it recursively enumerable} if there is an algorithm which on input $w\in S^*$ has output ``Yes'' iff $w\in R$. For $w\notin R$ the algorithm may or may not stop. Equivalently, there is an algorithm that enumerates the members of $S$.

(ii) A set $R\subset S^*$ is called {\it recursive} if both $R$ and $S^*\setminus R$ are recursively enumerable. Equivalently, there is an algorithm which on input $w\in S^*$ has output ``Yes'' if $w\in R$ and ``No'' if $w\notin R$.
\end{definition}

\begin{proposition}\label{PropositionConjugateSubshift} Let $(X,T)$ be  a minimal subshift. There exists a minimal subshift $(Y,S)$ topologically conjugate to $(X,T)$ such that  every word $w\in L_5(Y)$ has no repeated letters. Furthermore, $L(X)$ is recursive if and only if $L(Y)$ is recursive.
\end{proposition}
\begin{proof} Choose the least $n_0\geq 1$ such that
\begin{equation}\label{EqDisjointSets}T^i[.w]\cap [.w] = \emptyset\mbox{ for }i=1,\ldots,4,\;w\in L_{n_0}(X).\end{equation}
Such a number $n_0$ can be  found effectively by inductively testing words in $L_n(X)$ for increasing $n\geq 1$. The minimality of $(X,T)$ ensures that such an $n_0$ exists. We will treat the words in $L_{n_0}(X)$ as a new alphabet $B$. Define $\pi : X\rightarrow B^\mathbb Z$ as follows

$$\pi(\{x_n\}) = \cdots(x_{-n_0}x_{-n_0+1}\cdots x_{-1}).(x_0x_1\cdots x_{n_0-1})(x_1x_2\cdots x_{n_0})\cdots .$$  Set $Y = \pi(X)$. Denote by $S$  the left shift on $Y$. It is straightforward to check that $(X,T)$ and $(Y,S)$ are topologically conjugate. So $(Y,S)$ is a minimal subshift. Equation (\ref{EqDisjointSets}) guarantees us that $(Y,S)$ satisfies the desired property.

If $L(X)$ is recursive, taking the $\pi$-preimages, we can decide whether $w\in B^*$ lies in $L(Y)$.  Conversely, suppose $L(Y)$ is recursive. Any word $w=w_0\ldots w_{n-1}\in A^n$, $n\geq n_0$, can be uniquely represented as $$(w_0w_1\cdots w_{n_0-1})(w_1w_2\cdots w_{n_0})\cdots(w_{n-n_0}w_{n-n_0+1}\cdots w_{n-1}).$$ Therefore, $w\in L(X)$ iff the representation  of $w$ lies in $L(Y)$. If $w\in A^n$, $n<n_0$, consider all words in $A^{n_0}$ containing $w$ as a subword. Then $w\in L(X)$ iff $w$ is a subword of a word from $L_{n_0}(X)$.
 \end{proof}

\subsection{Kakutani-Rokhlin partitions}\label{SectionKR_Partitions}

In this section we explain the concept of Kakutani-Rokhlin partitions of minimal systems.

\begin{definition}  A family $\xi$ of disjoint clopen sets of the form $$\xi =\{B,TB,\ldots, T^{n-1}B\}$$  is called a {\it $T$-tower with  base $B$ and height $n$}. A clopen partition of $\Omega$ of the form $$\Xi = \{T^iB_v : 0\leq i\leq h_v-1, \; v=1,\ldots, q\}$$ is called a {\it Kakutani-Rokhlin partition}.  The sets $\{T^iB_v\} $
 are called {\it atoms} of the partition $\Xi$.
\end{definition}

Fix an arbitrary clopen set $B\subset \Omega$. Define a function $t_B : B \rightarrow \mathbb N$ by setting $$t_B( \omega) = \min \{ k \geq 1 : T^k \omega\in B\}.$$ Using the minimality of $T$, one can show  that the function $t_B$ is well-defined, bounded, and  continuous. Denote by $K$ the set of all integers $k\in \mathbb N$ such that the set $B_k = \{\omega\in B : t_B(\omega) = k\}$ is non-empty.    It follows from the continuity of $t_B$ that the set $K$ is finite and $B= \bigsqcup _{k\in K} B_k$ is a clopen partition. Using the definition of  $t_B$, one can check that the family  $$ \{T^i B_k : 0\leq i\leq k-1,\;k\in K\}$$ consists of disjoint sets and that
$$\Omega = \bigsqcup _{k\in K}\bigsqcup_{i = 0}^{k-1}T^iB_k.$$ Thus, $\Xi$  is a Kakutani-Rokhlin partition of $\Omega$.
\begin{definition}

The union of the sets $B(\Xi) = \bigsqcup_{k\in K}B_k$ is called the {\it base of the partition $\Xi$} and the union of the top levels $H(\Xi) = \bigsqcup_{k\in K} T^{k-1}B_k$ is called the {\it top or roof of the partition}. \end{definition}

 In the case of symbolic systems, Kakutani-Rokhlin partitions can be given purely combinatorial interpretation through the concept of return words.   The technique of return words  was developed by F.~Durand  and his coauthors  \cite{DurandHostSkau,Durand:1998}. The return words play  the same  role as the sets $B_k$ in the definition of the first-return map above.

\begin{definition}\label{DefinitionReturnWords} Let $u,v \in L(\Omega)$ be such that $uv\in L(\Omega)$. A word $w\in L(\Omega)$ is called a {\it return word} to $u.v$ if it satisfies the following properties:

\begin{itemize}
\item[(i)] $uwv\in L(\Omega)$;

\item[(ii)]  $v$ is a prefix of $wv$ and $u$ is a suffix of $uw$;

\item[(iii)] the word $uwv$ contains only two occurrences of $uv$.
 \end{itemize}

In other words, we find two consecutive occurrences of the pair $u.v$ in a sequence from $\Omega$, then  the word spanning the first occurrence of $v$ to the second occurrence of $u$ is called a return word. The following figure illustrates this concept. Note that, in general, the two occurrences of $uv$ can {\it overlap}.

\begin{figure}[hb]
\centering
\includegraphics[width = 5cm ]{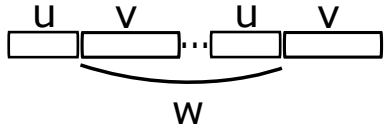}
\caption{Return word to $u.v$.}
\end{figure}

The minimality of the system $(\Omega,T)$ implies that every word appears in every sequence of $\Omega$ with bounded gaps, see, for example,  \cite{Durand:1998}. Thus, the set of return words $R_{u.v}$ to $u.v$ is  non-empty and finite.
  \end{definition}

 Fix an arbitrary sequence $ \omega \in \Omega$. Denote by $R_n$ the set of return words to $\omega[-n,-1].\omega[0,n]$ .
Set $u =  \omega [-n,-1]$ and $v =  \omega[0,n]$. By definition of return words, the cylinder sets $[u.rv]$, $r\in R_n$, are contained in $[u.v]$. Furthermore, $\{[u.rv] : r\in R_n\}$ form a clopen partition of $[u.v]$, which, in particular, implies that $R_n$ is finite for all $n$. Notice that if $y\in [u.rv]$, $r\in R_n$, then the smallest $k>0$ (termed the {\it first return time}) such that   of $T^{k}y\in [u.v]$ is $|r|$, by the property (iii) of return words. This implies that $$\Xi_n = \{T^i[u.rv] : r\in R_n,\; i=0,\ldots,|r|-1\}$$ is a clopen Kakutani-Rokhlin partition of $\Omega$ with base $[u.v]$.

We note that  for every $r\in R_{n+1}$ there is a unique decomposition $r = r_1\cdots r_k$ with $r_i\in R_n$, see the proof in  \cite[Proposition 2.6]{Durand:1998} and \cite[Corollary 18]{DurandHostSkau}. This implies that the partition $ \Xi_{n+1}$ refines $ \Xi_n$ and that the family $\{\Xi_n\}_{n\geq 1}$ generates the topology of $\Omega$. Since  $[u_n.v_n]$ is the base $()$ of the partition  $\Xi_n$, denoted by $B(\Xi_n)$,  we get that $\bigcap_{n\geq 1} B(\Xi_n) = \{\omega\}$. We summarize the properties of the KR partitions $\{\Xi_n\}$ in the following result.

\begin{proposition}\label{PropositionExistKRPartitions} Let $(\Omega,T)$ be a minimal subshift and $ \omega \in \Omega$ be an arbitrary point. Let $R_n$ be the set of return words to  $ \omega[-n,-1]. \omega[0,n]$. Then for every $n\geq 1$ the family $\Xi_n = \{T^i[\omega[-n,-1].r \omega[0,n]] : r\in R_n,\; i=0,\ldots,|r|-1\}$ is a clopen Kakutani-Rokhlin partition of $\Omega$. Furthermore,

\begin{enumerate}

\item the partitions $\{\Xi_n\}_{n\geq 1}$ generate the topology of $\Omega$;

\item the partition $\Xi_{n+1}$ refines $\Xi_n$ for every $n\geq 1$;

\item $\bigcap_{n\geq 1} B(\Xi_n) = \{\omega\}$ and $B(\Xi_{n+1})\subset B(\Xi_n)$ for every $n\geq 1$;

\item the height of the shortest tower in $\Xi_n$ goes to infinity as $n\to \infty$;

\item $\sup_{-n\leq i \leq n}\textrm{diam}(T^i(B(\Xi_n))) \to 0$ as $n\to \infty$.

\end{enumerate}
\end{proposition}


\subsection{Permutations}\label{SectionPermutations}
In this section we define a special kind of elements of $G_T$ termed  {\it permutations}.  Simplest examples of permutations arise from the canonical embeddings of symmetric groups  into $G_T$ defined as follows. Consider a clopen set $U$ such that the sets $\{U,TU,\ldots,T^{n-1}U\}$ are mutually disjoint. Consider elements $s_i\in G_T$ such that $s(x) = T(x)$ for $x\in T^iU$, $s(x) = T^{-1}(x)$ for $x\in T^{i+1}U$, and $s(x) = x$ elsewhere. Then the elements $<s_0,\ldots,s_{n-2}>$ generate the symmetric group $\textrm{Sym}(n)$.

  The KR partitions provide a natural framework for embedding direct products of copies of symmetric groups into the full group. Using sequences of partitions as in Proposition \ref{PropositionExistKRPartitions}, we will be able  to embed inductive limits of direct symmetric groups.

Let  $\{\Xi_n\}_{n\geq 1}$ be a sequence of KR partitions as in Proposition \ref{PropositionExistKRPartitions}.  Assume that for each $\Xi_n$ the minimum height of its towers is at least three. Set $u_n = \omega[-n,-1]$ and $v_n = \omega[0,n]$. Then for every $n\geq 1$, we have that $$\Xi_n = \{T^i[u_n.r v_n] : r\in R_n,\; i=0,\ldots,|r|-1\}.$$

\begin{definition}\label{DefinitionPermutation} Fix an integer $n\geq 1$.   We say that a homeomorphism $P\in G_T$ is an {\it permutation associated with the KR partition $\Xi_n$} if (1) its orbit cocycle $f_P$ is compatible (constant on atoms) with the partition $\Xi_n$ and (2) for any point $x\in T^i [u_n.r v_n]$ ($0\leq i\leq |r|-1$, $r\in R_n$) we have that $0\leq f_P(x) + i |r|-1$. The latter condition means that $P$ permutes atoms only within each tower without moving points over the top or the base of the tower. We will  call $P$ just a {\it permutation} when the partition $\Xi_n$ is clear from the context.

Note that each permutation $P$ can be uniquely factored into a product of permutations $P_1,\ldots,P_{|R_n|}$ such that $P_r$ acts only within the tower $$\xi_r^{(n)} = \{T^i[u_n.r v_n] : i=0,\ldots,|r|-1\},\; r\in R_n.$$
\end{definition}

The following figure illustrates a KR-partition consisting of three $T$-towers.  The leftmost ellipses represent the bases of the $T$-towers. The arrows give the directions in which points are headed under the action of $T$. Note that the points from the top of the towers (the rightmost ellipses) are sent back to one of the bases (it is not necessarily the base it originally came from).

\begin{figure}[ht]\label{FigurePermutations}
\centering
\includegraphics[width = 5cm ]{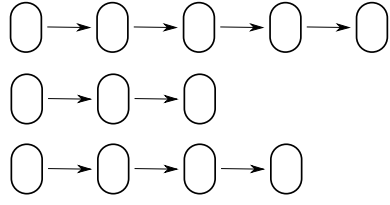}
\caption{A KR-partition with three $T$-towers}
\end{figure}

 Denote by $G_{T,\omega}^{(n)}$ the group of permutations built upon $\Xi_n$. It has a structure of the direct product of symmetric groups $\prod_{r\in R_n}\mathrm{Sym}(|r|)$. For example, for the KR-partition shown in the figure the group $G_{T,\omega}^{(n)}$  is isomorphic to $\mathrm{Sym}(5)\times \mathrm{Sym}(3) \times \mathrm{Sym}(4)$.

 Since $\Xi_{n+1}$ refines $\Xi_n$, we get that $G_{T,\omega}^{(n)} \subset G_{T,\omega}^{(n+1)}$. Denote by $G_{T,\omega} = \bigcup_{n\geq 1} G_{T,\omega}^{(n)}$ the group of all permutations associated with the sequence of KR partitions $\{\Xi_n\}_{n\geq 1}$.  For $r\in R_n$ and $0\leq i \leq |r|-1$, denote by $(u_n.rv_n,i)$ the set $T^i[u_n.rv_n]$. Denote by $\sigma_{(u_n.rv_n,i)}$ the homeomorphism  such that $$\sigma_{(u_n.rv_n,i)}(x) = Tx\mbox{ for }x\in (u_n.rv_n,i)\cup (u_n.rv_n,i+1)$$
 and $$\sigma_{(u_n.rv_n,i)}(x) = T^{-2}x\mbox{ for }x\in (u_n.rv_n,i+2).$$
 See Section \ref{SectionGenerators} for a detailed discussion of properties of  $\sigma_{(u_n.rv_n,i)}$.

\begin{proposition}\label{PropositionGroupPermutations} Let $\{\Xi_n\}_{n\geq 1}$ be a sequence of Kakutani-Rokhlin partitions corresponding to $\omega \in \Omega$ as in Proposition \ref{PropositionExistKRPartitions}.

(i) Then the group $G_{T,\omega}$  coincides with the group of elements in $G_T$ preserving the positive half-orbit of $\omega$ under $T$.

(ii) The commutator subgroup $G_{T,\omega}'$ is simple.

(iii) The group $G_{T,\omega}'$ is the increasing union of subgroups $\left(G_{T,\omega}^{(n)}\right)'$, $n\geq 1$.

(iv)  The commutator subgroup  $\left(G_{T,\omega}^{(n)}\right)' $ is  generated by
 $\{\sigma_{(u_n.rv_n,i)}\}$, $r\in R_n$, $i=0,\ldots,|r|-3$, and isomorphic to  the group
 $\prod_{r\in R_n}\mathrm{Alt}(|r|)$.

(v) For each $n\geq 1$,  the embedding   $\left(G_{T,\omega}^{(n)}\right)' \subset \left(G_{T,\omega}^{(n+1)}\right)'$  is given by the rule: for $r\in R_n$ and $0 \leq i \leq |r|-3$,
\begin{equation*}
 \sigma_{(u_n.rv_n, i)} = \prod_{r'\in R_{n+1}}\prod_{\substack{ 0\leq j\leq |r'|-3 \\ (u_{n+1}.r'v_{n+1}, j)  \subset (u_n.rv_n, i)}}\sigma_{(u_{n+1}.r'v_{n+1}, j)}.
 \end{equation*}

\end{proposition}
\begin{proof}
The first statement was established in \cite{GirdanoPutnamSkau:1999}, see also  the discussion before Proposition 5.2 in \cite{GrigorchukMedynets}. The proof of the second statement is due to Matui  \cite[Lemma 3.4]{Matui:2006}. The remaining conclusions of the proposition   follow  directly from the construction of the KR-partitions $\{\Xi_n\}$.
\end{proof}

The following result shows that permutations from $G_{T,\omega}$ almost completely determine the structure of TFGs. The proof can be found in \cite[Theorem 5.4]{GrigorchukMedynets}, see also \cite[Lemma 4.1]{Matui:2006}.

\begin{theorem}\label{TheoremPermutationProduct} Let $(\Omega,T)$ be a Cantor minimal system. Let $\omega,\omega'\in \omega$ be from distinct $T$-orbits. Then $G_T'= G_{T,\omega}'\cdot G_{T,\omega'}'$.
\end{theorem}

 \begin{corollary}\label{CorollaryProductTwoCommutators} (1) Every element of the commutator subgroup $G_T'$  is the product of at most two commutators in $G_T$.

 (2) Every element of $G_T'$ is the product of at most four involutions.
 \end{corollary}
 \begin{proof} (1) It immediately follows from the previous theorem and the fact that any even permutation in $\mathrm{Sym}(n)$ is a commutator \cite{Ore:1951}.

 (2) Notice that every permutation in $\mathrm{Sym}(n)$  is a product of two involutions. The result follows from Theorem \ref{TheoremPermutationProduct}.
 \end{proof}


\section{Generators}\label{SectionGenerators}

In \cite{Matui:2006} showed that the commutator subgroup $G_T'$ is finitely generated if and only if the system $(\Omega,T)$ is topologically conjugate to a subshift. In  the proof he  explicitly described a set of generators. In this section we reestablish Matui's result and spell out some of the relations satisfied by  the generators.

 Consider a clopen set $U$ such that $\{U, TU, T^2U\}$ are mutually disjoint.   Define $\sigma_U \in \G_T$ as follows:
$$\sigma_U(\omega) = \left\{\begin{array}{lll} T\omega & \mbox{if} & \omega\in U\cup T U \\
T^{-2}\omega & \mbox{if} & \omega \in T^2U\\
\omega & \mbox{if} & \mbox{otherwise }.
\end{array} \right.$$

Note that $\sigma_U = [\eta_U,T]\in G_T'$, where $\eta_U(x) = Tx$ if $x\in U$, $\eta_U(x) = T^{-1}U$ if $x\in TU$, and $\eta_U(x) = x$ elsewhere.  Elements $\sigma_U$ can be viewed as  cycles of length three  $(i,i+1,i+2)$ with set $U$ corresponding to $i$.

By definition, $\sigma_\emptyset $  is the group identity. In particular,   $w\notin L(\Omega)$ if and only if  $\sigma_{[.w]}=1$, the identity map. Recall that for for $w\in L(\Omega)$ and $i\in \mathbb Z$, $(w,i) = T^{i}[.w]$ and, hence,
$$\sigma_{(w,i)} = \sigma_{T^i[.w]}.$$

 The following lemma is due to Matui \cite[Lemma 5.3]{Matui:2006}. Matui used it to show that every element  of the form $\sigma_U$, where $U$ is a clopen set, can be generated by elements $\sigma_{[a.bc]}$, $abc\in L_3(\Omega)$.  He then used the simplicity of the group $G_T'$ to show that the group $G_T'$ is generated by $\sigma_U$, when $U$ runs over all clopen sets, and, as a result, that $G_T'$ is finitely generated, see Theorem \ref{TheoremGenerators}.

 It will be  convenient to use the following modified version of the commutator operation
 \begin{equation}r*s := sr^{-1}s^{-1}r \mbox{ for }r,s\in G_T.\end{equation}
 Observe that the operation ``$*$''  is not associative. Recall that we assume that all minimal systems  satisfy the condition ($\dag$), that is, that  {\it no word in $L_5(\Omega)$ has repeated letters.}

\begin{lemma}\label{LemmaMatui} Let $U,V$ be clopen subsets of $\Omega$ such that $U$, $T(U)$, $T^2(U)\cup  V$, $T(V)$, and $T^2(V)$ are mutually disjoint. Then $\sigma_U \ast \sigma_V = \sigma_{T(U)\cap T^{-1}(V)}$.
\end{lemma}

The following figure illustrates the assumptions of  Lemma \ref{LemmaMatui}. The shaded region represents the orbit of $\sigma_{T(U)\cap T^{-1}(V)}$.

\begin{figure}[ht]
\centering
\includegraphics[width = 5cm ]{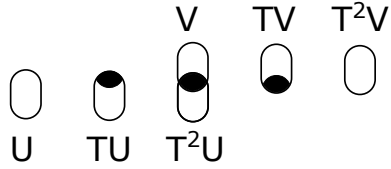}
\caption{Relative positions of sets $U$ and $V$.}
\end{figure}

%


\begin{definition}\label{DefinitionDisjoint} (1) Let $(w,i)$ be a cylinder set. By   a  {\it cylinder partition} of $(w,i)$ we mean a partition of it  into cylinder sets of the form $(v,j)$.

(2)  We will say that two clopen sets $U$ and $V$  are {\it $3$-disjoint} if   $(U\cup TU \cup T^2 U)\cap (V\cup TV\cup T^2 V) =\emptyset$. Equivalently, $U$ and $V$ are $3$-disjoint if  $\sigma_U$ and $\sigma_V$ have disjoint supports.
 \end{definition}

\begin{remark}\label{Remark3disjoint} (1) Let $U\subset \Omega$ be a clopen set such that $\{U,TU,T^2U\}$ are pairwise disjoint. Suppose that $C$ and $D$ are disjoint subsets of $U$. Then $C$ and $D$ are 3-disjoint.

(2) Let $u,v\in L(\Omega)$ and $\Xi$ be the Kakutani-Rokhlin partition associated to the set of return words $R_{u.v}$, see Section \ref{SectionKR_Partitions} for details. By construction, $\Xi$ consists of the  disjoint sets $$\{T^i[u.rv] : r\in R_{u.v},\;0\leq i \leq |r|-1 \}.$$ Suppose that the height of each tower in $\Xi$ is at least three. Then for distinct $r,r'\in R_n$ and arbitrary $0\leq i\leq |r|-3$ and $0\leq i'\leq |r'|-3$,  the sets  $$\bigcup_{j = i}^{i+2}T^i[u.rv] \mbox{ and }\bigcup_{j = i'}^{i'+2}T^{i'}[u.r'v]$$ are disjoint. Hence,  $T^i[u.rv]$ and $T^{i'}[u.r'v]$ are 3-disjoint.
 \end{remark}

 \begin{proposition}\label{PropositionDecidabilityDisjointness} Suppose that the language $L(\Omega)$ of the subshift $(\Omega,T)$ is recursive. Let $(w,i)$ and $(v,j)$ be   cylinder sets. Then

  (1) The inclusion $(v,j) \subset (w,i)$  is algorithmically verifiable.

 (2) The $3$-disjointness of $(w,i)$ and $(v,j)$ is algorithmically verifiable.

  (3)   The set of cylinder partitions of $(w,i)$ is a recursive subset of $\mathcal{FS}(A^*\times \mathbb Z)$, where $\mathcal{FS}(A^*\times \mathbb Z)$ is the set of  all finite subsets of $A^*\times\mathbb Z$.
 \end{proposition}
 \begin{proof} (1) Note that $(v,j)\subset (w,i)$  if and only if $(v,0) \subset (w,i-j)$. To verify the latter inclusion, we use the recursiveness of the language $L(\Omega)$ to iterate over  words in $L(\Omega)$ of the form $lr$, $l,r\in L(\Omega)$ such that $|l| = |r| = \max\{|w|,|v|\}+|i-j|$ and $v$ is a prefix of $r$. If every such word $lr$ contains  $w$ as a subword at the position $|l|-(i-j)$, then $(v,0)$ is a subset of $(w,i-j)$. For otherwise, it is not.

 (2) The $3$-disjointness of $(w,i)$ and $(v,j)$ is equivalent to the $3$-disjointness of $(w,i-j)$ and $(v,0)$.    The cylinder sets  $(w,i-j)$ and $(v,0)$ are not 3-disjoint if and only if there is a word $lr\in L(\Omega)$ such that $|l|  = |r| = \max\{|v|,|w|\} + |j-i|+3$; $w$ appears in $lr$ at the position $|l| - (i-j)$ and $v$ occurs at the position $|l|+k$ for some $-3\leq k \leq 2$. Iterate over all such words $lr$ to verify these conditions.

  (3) Note  that partitions of $(w,i)$ are obtained from  partitions of $(w,0)$ by shifting them by $i$.   Thus we  need to build an algorithm that decides whether  a set $S \in \mathcal{FS}(A^*\times \mathbb Z)$ forms a partition of $(w,0)$. First, we use the recursiveness of the language to verify that $v\in L(\Omega)$, for $(v,j)\in S$.
  Then we use the algorithm from (1) above to  check whether every element of  $S$  is a subset of $(w,0)$.

  Finally, we need to check that the elements of $S$ are disjoint and  cover all of $(w,0)$.  Set $M = \max_{(v,j)\in S}\{|v|+|j| + |w|\}$. Using (1) above and recursiveness of the language, generate all words $\mathcal E$ in $\Omega$ of length $2M$ with $(v,M)\subset (w,0)$ for $v\in \mathcal E$. Then $S$ is a partition of $(w,j)$ iff for every $v\in \mathcal E$ there is exactly one element $(r,k)\in S$ with $(v,M)\subset (r,k)$. Thus, the family $S$ is a cylinder partition of $(w,0)$ iff we get ``yes'' at every step of the algorithm.

 \end{proof}

The following result shows how one can use elements $\sigma_{[a.bc]}$, $abc\in L_3(\Omega)$, to generate arbitrary elements of the form $\sigma_U$, where $U$ is a clopen set.

\begin{proposition}\label{PropositionGroupWords} (1) For any cylinder set $(v,i)$, the sets $(v,i)$, $(v,i+1)$, $(v,i+2)$ are pairwise disjoint. In particular, the element $\sigma_{(v,i)}$ is well-defined.

(2) Let $U$ be a clopen set such that the sets $\{U,TU,T^2U\}$ are pairwise disjoint. Then for any  clopen partition  $\mathcal C$  of $U$, we have that   $$\sigma_{U} = \prod_{D\in \mathcal C}\sigma_{D}$$ and the elements in this product commute. In particular, for  $a,b,c\in A$, we have  that $$\sigma_{[a.]} = \prod_{b,c\in A}\sigma_{[a.bc]}\mbox{ and }\sigma_{[.bc]} = \prod_{a\in A}\sigma_{[a.bc]}.$$

(3) For $w\in A^n$, $n\geq 4$, we have that
\begin{equation}\label{EquationConvolutions1}
\sigma_{[w_0|w_1\cdots w_{n-1}]} =  \sigma_{[.w_0w_1]} \ast (\sigma_{[.w_1w_2]}\ast \cdots \ast(\sigma_{[.w_{n-4}w_{n-3}])} \ast \sigma_{[w_{n-3}|w_{n-2}w_{n-1}}])\cdots).
\end{equation}

(4) For  $w\in A^n$, $n\geq 4$, and $2\leq k \leq n-1$, we have that
\begin{equation}\label{EquationConvolutions2}
\sigma_{[w_0\cdots w_{k-1}|w_k\cdots w_{n-1}]}  =   (\cdots ((\sigma_{[w_0|w_1\cdots w_{n-1}]}
 \ast  \sigma_{[w_2.]}) \ast \sigma_{[w_3.]}) \ast \cdots \ast \sigma_{[w_{k}.]}).
\end{equation}
\end{proposition}
\begin{proof}  (1) The assumption that no word of length five has repeated letters  implies that for every $a\in A$, the cylinder sets $\{(a,0),(a,1),(a,2)\} $
are pairwise disjoint. For otherwise, the language $L(\Omega)$ would contain a word of the form $a \Box  a$ or $a\Box \Box a$, where $\Box$ is a placeholder for an arbitrary letter of $A$, which is a contradiction.    Thus, $\sigma_{(a,i)}$ is well-defined. Noticing that every cylinder set is a subset of $(a,i)$ for some $a\in A$ and $i\in \mathbb Z$, we get the proof of the first statement.

(2) The second statement immediately follows from the definition of $\sigma_U$ and the statement (1) above.

(3) Consider a word $abcd \in L_4(\Omega)$. We claim  that the sets $U = [b|cd]$ and $V = [.ab]$ satisfy the assumptions of Lemma \ref{LemmaMatui}. Indeed, suppose that $(U\cup TU \cup T^2U) \cap V \neq \emptyset$.   Then at least one of the sets $[b|cd]\cap [.ab]$, $[bc|d]\cap [.ab]$, and $[bcd.]\cap [.ab]$ is non-empty.  Now if  $(V\cup TV \cup T^2V) \cap U \neq \emptyset$, then  one of the sets $[.ab]\cap [b|cd]$, $[.\Box ab]\cap [b|cd]$, $[.\Box \Box ab]\cap [b|cd]$ is non-empty. In either case, there is a word of length five with at least two occurrences of the letter $b$, which contradicts the assumptions.

Thus, by Lemma \ref{LemmaMatui} for any word $abcd\in L_4(\Omega)$, we have that $$\sigma_{[.ab]} \ast \sigma_{[b|cd]} = \sigma_{[a|bcd]}.$$ Equation (\ref{EquationConvolutions1}) for $w\in L(\Omega)$ is established by the induction on the length of $w$.

To complete the proof, we need to  show that both  sides of Equation (\ref{EquationConvolutions1}) are trivial whenever $w\notin L(\Omega)$.  Fix a word $w\notin L(\Omega)$.  Write it out as $w = w_0 \cdots w_{n-1}$, $w_i\in A$, $i=0,\ldots,n-1$. Choose the largest $k\leq n-1$ so that $w_{k}\cdots w_{n-1} \in L(\Omega)$. If $w_{n-3}w_{n-2}w_{n-1} \notin L(\Omega)$, then $\sigma_{[w_{n-3}|w_{n-2}w_{n-1}}]) = 1$ and both  sides of Equation (\ref{EquationConvolutions1}) become trivial. Thus, we can assume that $k\leq n-3$.

Set $U = [.w_{k-1}w_k]$ and $V = [w_k|w_{k+1}\cdots w_{n-1}]$.  It follows from the assumption that no word of length five in $L(\Omega)$ has repeated letters that the sets $U$ and $V$ satisfy Lemma \ref{LemmaMatui}. Applying Equation (\ref{EquationConvolutions1}) to $\sigma_{[w_k|w_{k+1}\cdots w_{n-1}]}$ we get that
\begin{eqnarray*}\sigma_{[w_{k-1}|w_k\cdots w_{n-1}]} & = & \sigma_{[.w_{k-1}w_k]} \ast \sigma_{[w_k|w_{k+1}\cdots w_{n-1}]} \\
& = &  \sigma_{[.w_{k-1}w_k]}\ast( \cdots \ast(\sigma_{[.w_{n-4}w_{n-3}])} \ast \sigma_{[w_{n-3}|w_{n-2}w_{n-1}}])\cdots)\end{eqnarray*}
Since  $\sigma_{[w_{k-1}|w_k\cdots w_{n-1}]} =1$, Equation (\ref{EquationConvolutions1}) holds trivially for $\sigma_{[w_{0}|w_1\cdots w_{n-1}]}$.
x

 (4) We note that it suffice to verify Equation (\ref{EquationConvolutions2}) only for $w\in L(\Omega)$. Set $U = [w_0|w_1\cdots w_{n-1}]$ and $V = [.w_2]$. The sets $U$ and $V$ satisfy the assumptions of Lemma \ref{LemmaMatui}. Hence $$\sigma_{[w_0w_1|w_2\cdots w_{n-1}]} =  \sigma_{[w_0 | w_1\cdots w_{n-1}]} \ast \sigma_{[.w2]}.$$ The proof follows by induction.
\end{proof}

The following result summarizes the properties of $\sigma_U$'s and lists some of the relations they satisfy.  It will be shown in Theorem \ref{TheoremMain} that this list of relations is, in fact, complete.

\begin{corollary}\label{CorollarySigmaRelations} Let $(\Omega,T)$ be a minimal subshift satisfying the condition ($\dag$). Then the elements $\{\sigma_{(w,i)}\}$ satisfy the following relations: for every $w,v\in L(\Omega)$, $i,j\in \mathbb Z$, and a cylinder partition $C$ of $(w,i)$,
\begin{flalign*}
& \left(\sigma_{(w,i)}\right)^3=1   \\
& \left(\sigma_{(w,i)}\cdot \sigma_{(w,i+1)}\right)^2 = 1   \\
& \sigma_{(w,i+1)} = \sigma_{(w,i)} \ast \sigma_{(w,i+2)}  \\
& \sigma_{(w,i)} = \prod_{(s,k) \in C}\sigma_{(s,k)}
\end{flalign*}
and
\begin{equation*} [\sigma_{(w,i)},\sigma_{(v,j)}] = 1,\mbox{ whenever }(w,i),(v,j)\mbox{ are 3-disjoint}.
\end{equation*}
\end{corollary}

The following result follows from Proposition \ref{PropositionGroupWords} and allows us to relate the word problem in full groups to the decidability of the language $L(\Omega)$.

\begin{corollary}\label{CorollaryWordIdentity} Suppose that $(\Omega,T)$ is a minimal subshift over a finite alphabet $A$ satisfying the condition ($\dag$). Let $w = w_0\cdots w_{n-1}\in A^n$, $n\geq 4$. Then $w\in L(\Omega)$ iff
$$\sigma_{[.w_0w_1]} \ast (\sigma_{[.w_1w_2]}\ast \cdots \ast(\sigma_{[.w_{n-4}w_{n-3}]} \ast \sigma_{[w_{n-3}.w_{n-2}w_{n-1}}])\cdots) = 1\mbox{ in }G_T'.$$
\end{corollary}

  We now present an alternate proof of Theorem 5.4 in \cite{Matui:2006} that does not rely on the simplicity of the commutator subgroup $G_T'$.

\begin{theorem}[\cite{Matui:2006}]\label{TheoremGenerators} Suppose that $(\Omega,T)$ is a minimal subshift satisfying the condition ($\dag$). Then the commutator subgroup $G_T'$ is generated by $\sigma_{[a.bc]}$, $abc\in L_3(\Omega)$.
\end{theorem}
\begin{proof}  Consider an arbitrary element $Q\in G_T'$.  Fix two points $w,w'\in \Omega$ taken from distinct $T$-orbits.  It follows from   \cite[Theorem 5.4]{GrigorchukMedynets}, see also Theorem \ref{TheoremPermutationProduct}, that there exist $P\in G_{T,\omega}'$ and $P'\in G_{T,\omega'}'$ such that $Q = PP'$.  To establish the result, it suffices to prove that both  $P$ and $P'$ are generated by $\sigma_{[a.bc]}$, $abc\in L_3(\Omega)$.

By definition of $P$, there exists a Kakutani-Rokhlin partition $\Xi = \{\xi_v\}$ such that $P = \prod_v P_v$, where $P_v$ is an even  permutation  acting within the $T$-tower $\xi_v = \{U_v,TU_v,\cdots, T^{h_v-1}U_v\}$. Here $h_v$ is the height of  $\xi_v$ and $U_v$ is its base.  Proposition \ref{PropositionGroupWords} implies that the elements $\sigma_{[a.bc]}$ can be used to generate elements of the form $\sigma_{T^iU_v}$. Since $P$ is even, it is generated by 3-cycles of the form $\sigma_{T^iU_v}$. This shows that $P$ is generated by $\sigma_{[a.bc]}$, $abc\in L_3(\Omega)$. The proof for $P'$ is analogous.
\end{proof}

\begin{corollary}\label{CorollaryN_Generators} Let $(\Omega,T)$ be an arbitrary minimal subshift. The commutator subgroup $G_T'$ is generated by $\sigma_{(w,1)}$, $w\in L_n(\Omega)$, for some $n\geq 3$.  If the language $L(\Omega)$ is recursive, then $n$ is effectively computable.
\end{corollary}
\begin{proof} Using Proposition \ref{PropositionConjugateSubshift}, we can find a minimal subshift $(\Omega',T')$ satisfying the condition $(\dag)$ that is conjugate to $(\Omega,T)$. Let $\pi$ and $n_0>0$  be as in the proof of
Proposition \ref{PropositionConjugateSubshift}.  The homeomorphism $\pi$ gives rise to an isomorphism of the commutator subgroups $G_T'$ and $G_{T'}'$. Note that by definition of $\pi$,  $$\pi((w,i)) = (\rho(w),i)\mbox{ for every }w\in L_{n_0+2}(\Omega),$$ where  $\rho(w) = (w_0\cdots w_{n_0-1})(w_1\cdots w_{n_0})(w_1\cdots w_{n_0+1})$ for $w= w_0w_1\cdots w_{n_0+1}\in L_{n_0+2}(\Omega)$.
Therefore, $$\pi\circ \sigma_{(w,1)}\circ \pi^{-1} = \sigma_{(\rho(w),1)}\mbox{ for every }w\in L_{n_0+2}(\Omega).$$ Applying Theorem \ref{TheoremGenerators}, we conclude that the elements $\sigma_{(w,1)}$, $w\in L_{n_0+2}(\Omega)$, generate the group $G_T'$.
\end{proof}


%
%

\section{Defining Relations}\label{SectionRelations}

In this section we describe the set of defining relations for the commutator subgroups of topological full groups.  In the following theorem we use the symbols $x_{(w,i)}$ indexed by cylinder sets $(w,i)$ as a basis for the free group.

%
%

\begin{theorem}\label{TheoremMain}   Let $(\Omega,T)$ be a minimal subshift over a finite alphabet.
 (1) There exists $n\geq 3$ such that the commutator subgroup  $G_T'$ is isomorphic to the group $\Gamma_\Omega$ generated by  $$<x_{(w,k)}, w\in L(\Omega),\; |w|\geq n,\; k\in \mathbb Z>,$$  subject to the following relations: for every $w,v\in L(\Omega)$, $|w|,|v|\geq n$, $i,j\in \mathbb Z$,  and a cylinder partition $C$ of $(w,i)$,
\begin{flalign}
& \left(x_{(w,i)}\right)^3=1   \label{EqRel1} \\
& \left(x_{(w,i)}\cdot x_{(w,i+1)}\right)^2 = 1   \label{EqRel2} \\
& x_{(w,i+1)} = x_{(w,i)} \ast x_{(w,i+2)}  \label{EqRel3} \\
& x_{(w,i)} = \prod_{(s,k) \in C}x_{(s,k)} \label{EqRel4}
\end{flalign}
and
\begin{equation} \label{EqRel5} [x_{(w,i)},x_{(v,j)}] = 1,\mbox{ whenever }(w,i),(v,j)\mbox{ are 3-disjoint}.
\end{equation}

(2)  The isomorphism  $\psi : \Gamma_\Omega \rightarrow G_T'$ is implemented by the rule $\psi(x_{(w,k)}) = \sigma_{(w,k)}$ for all $w\in L(\Omega)$, $k\in \mathbb Z$, $|w|\geq n$.
 In particular, the group $\Gamma_\Omega$ is generated by the elements $<x_{(w,1)}$, $w\in L_n(\Omega)>$.

 (3) If the system $(\Omega,T)$ satisfies the condition $(\dag)$, then $n = 3$.
\end{theorem}
\begin{proof}  Let $n$ be as in the proof of  Corollary \ref{CorollaryN_Generators}. Then $(\Omega,T)$ is isomorphic to an $n$-fold minimal subshift satisfying ($\dag$). Thus, without loss of generality, we can assume that the system $(\Omega,T)$ satisfies the condition $(\dag)$ and  the group $G_T'$ is generated by $\sigma_{(abc,1)}$, $abc\in L_3(\Omega)$. Hence, it suffices to establish the theorem for  $n = 3$.

 Denote by $N$ the normal closure of the relations (\ref{EqRel1})-(\ref{EqRel5}) in the free group $F(S)$ with  basis $S = <x_{(w,k)}, w\in L(\Omega),\; |w|\geq 3,\; k\in \mathbb Z>.$
Denote the quotient group $F(S)/N$ by $\Gamma$.  To simplify the notation we will be skipping the bar in  $\bar x_{(w,k)}\in \Gamma$.

  Consider  the map $\psi : \Gamma \rightarrow G_T'$ such that  $\psi(x_{(w,k)}) = \sigma_{(w,k)}$ for every $(w,k)\in L(\Omega)\times \mathbb Z$. Extend $\psi$ to any element of $\Gamma$ by juxtaposition.
  The fact that the generators $\{\sigma_{(w,k)}\}$ of $G_T'$ satisfy  the relations (\ref{EqRel1}) -- (\ref{EqRel5}) (see Corollary \ref{CorollarySigmaRelations}) implies that
   $\psi$ is a well-defined homomorphism of the groups. Note that this, in particular, means that the group $\Gamma$ is non-trivial.

   Fix a sequence $\omega\in \Omega$.  For each $n$ set $u_n = \omega[-n,-1]$ and $v_n = \omega[0,n]$.   Denote by $R_n$ the set of return words to $u_n.v_n$.
Denote by $\Gamma_\omega$ the subgroup of $\Gamma$ generated by the elements  \begin{equation}\label{EqGenGx}\left <x_{(u_n.rv_n, i)} : r\in R_n,\; 0\leq i \leq |r| - 3,\; n\geq 1\right >,\end{equation} where $(u.rv,i) = T^i[u.rv]$, cf. Section \ref{SectionPermutations}. Note that $\psi(\Gamma_\omega) = G_{T,\omega}'$ (Proposition \ref{PropositionGroupPermutations}). Hence $\Gamma_\omega$ is non-trivial. In the following three lemmas we will prove that $\Gamma_\omega$ and $G_{T,\omega}'$ are isomorphic.

\begin{lemma}[\cite{VershikVsemirov:2008}]\label{LemmaAlternaingGroupRelations} For $n\geq 5$, the group $\textrm{Alt}(n)$ is isomorphic to the group generated by $y_0$,..,$y_{n-3}$ satisfying the relations

 \begin{equation*} y_i^3=1,\quad 0=1,\ldots,n-3,
 \end{equation*}
 \begin{equation*}
(y_i\cdot y_{i+1})^2 = 1, \quad  i = 0,\ldots, n-4,
 \end{equation*}
 \begin{equation*}
[y_i,y_j] = 1, \quad i,j=0,\ldots, n-3, \; |i-j| >2,
 \end{equation*}
 \begin{equation*}
y_{i+1} = y_i\ast y_{i+2}, \quad i=0,\ldots, n-5.
 \end{equation*}

The isomorphism is given by the mapping $y_i\mapsto (i,i+1,i+2)$, $i=0,\ldots, n-3$.
\end{lemma}

\begin{lemma}\label{LemmaGTW_presentation}   The group $G_{T,\omega}'$ is completely determined by the generators \begin{equation*}\left <\sigma_{(u_n.rv_n, i)} : r\in R_n,\; 0\leq i \leq |r| - 3,\; n\geq 1\right >\end{equation*}  and the relations:  for each $n\geq 1$ and $r\in R_n$,
 \begin{equation}\label{EqPermRel1} (\sigma_{(u_n.rv_n, i)})^3=1,\quad 0=1,\ldots,n-3,
 \end{equation}
 \begin{equation}\label{EqPermRel2}
(\sigma_{(u_n.rv_n, i)}\cdot \sigma_{(u_n.rv_n, i+1)})^2 = 1, \quad  i = 0,\ldots, |r|-4,
 \end{equation}
 \begin{equation}\label{EqPermRel3}
[\sigma_{(u_n.rv_n, i)},\sigma_{(u_n.rv_n, j)}] = 1, \quad i,j=0,\ldots, |r|-3, \; |i-j| >2,
 \end{equation}
 \begin{equation}\label{EqPermRel4}
\sigma_{(u_n.rv_n, i+1)} = \sigma_{(u_n.rv_n, i)}\ast \sigma_{(u_n.rv_n, i+2)}, \quad i=0,\ldots, |r|-5,
 \end{equation}
for each $n\geq 1$ and  distinct $r,r'\in R_n$,
\begin{equation}\label{EqPermRel5}
[\sigma_{(u_n.rv_n, i)},\sigma_{(u_n.r'v_n, j)}] = 1, \quad 0\leq  i \leq |r|-3,\; 0\leq j\leq |r'|-3,
\end{equation}
and for each $n\geq 1$, $r\in R_n$, and $0 \leq i \leq |r|-3$,
\begin{equation}\label{EqPermRel6}
 \sigma_{(u_n.rv_n, i)} = \prod_{r'\in R_{n+1}}\prod_{\substack{ 0\leq j\leq |r'|-3 \\ (u_{n+1}.r'v_{n+1}, j)  \subset (u_n.rv_n, i)}}\sigma_{(u_{n+1}.r'v_{n+1}, j)}
\end{equation}
\end{lemma}
\begin{proof} Denote by $Q$ the group with presentation $<\bigcup_n S_n,R>$, where $S_n = \left <\sigma_{(u_n.rv_n, i)} : r\in R_n,\; 0\leq i \leq |r| - 3\right >$ and $R$ is the set of relations (\ref{EqPermRel1})--(\ref{EqPermRel6}).  Note that the group $\left(G_{T,\omega}\right)'$ satisfies the relations (\ref{EqPermRel1}) -- (\ref{EqPermRel6}) (Proposition \ref{PropositionGroupPermutations} and Corollary \ref{CorollarySigmaRelations}). Hence, $Q$ maps onto $G_{T,\omega}$ and $Q$ is non-trivial.
Denote by $Q_n$ the subgroup of $Q$ generated by $S_n$. Note that $\bigcup Q_n = Q$.

By Proposition \ref{PropositionGroupPermutations}, the group $G_{T,\omega}'$ is the increasing union of finite subgroups $(G_{T,\omega}^{(n)})'$ isomorphic to the direct product  of alternating groups $\{\textrm{Alt}(|r|)\}$, $r\in R_n$. Each subgroup $(G_{T,\omega}^{(n)})'$ is generated by $\{ \sigma_{(u_n.rv_n, i)}\}$, $r\in R_n$, $i=0,\ldots,|r|-3$ (Proposition \ref{PropositionGroupPermutations}).

Note that for each $r\in R_n$ the element  $\sigma_{(u_n.rv_n, i)}$ corresponds to the permutation $(i,i+1,i+2)$ acting within the tower associated with the return word
$r$.   Thus, in view of Lemma \ref{LemmaAlternaingGroupRelations} the relations (\ref{EqPermRel1})--(\ref{EqPermRel4}) give a presentation for the copy of $\textrm{Alt}(|r|)$ in
$(G_{T,\omega}^{(n)})'$.

For distinct $r,r'\in R_n$ (distinct $T$-towers) and $0\leq i\leq |r|- 3$, $0\leq i'\leq |r'|- 3$, the clopen sets $(u_n.rv_n, i)$ and $(u_n.r'v_n, i')$ are 3-disjoint (Remark \ref{Remark3disjoint}). This shows that the relation (\ref{EqPermRel5}) holds in $(G_{T,\omega}^{(n)})'$. Observe that this relation is responsible for the direct product structure in $(G_{T,\omega}^{(n)})'$. Thus, we conclude that  $(G_{T,\omega}^{(n)})'$ is completely determined by the generators
$$\left<\sigma_{(u_n.rv_n, i)} : r\in R_n,\; 0\leq i \leq |r| - 3\right >$$
and the relations  (\ref{EqPermRel1})--(\ref{EqPermRel5}).  It follows that there exists a natural epimorphism $\psi_n: (G_{T,\omega}^{(n)})' \rightarrow Q_n$.

The relation (\ref{EqPermRel6}) allows us to express the generators of $(G_{T,\omega}^{(n)})'$ through the generators of $(G_{T,\omega}^{(n+1)})'$. Note that $$\psi_{n+1}|_{(G_{T,\omega}^{(n)})'} = \psi_n.$$
Hence the sequence of homomorphisms $\{\psi_n\}$ naturally extends to an epimorphism $\psi : G_{T,\omega}' \rightarrow Q$.  Since $G_{T,\omega}'$ is simple (Proposition \ref{PropositionGroupPermutations}), we conclude that $\psi$ must be a group isomorphism.
\end{proof}

\begin{lemma}\label{LemmaPresentationPermutationGroup}
The group $G_{T,\omega}'$ is isomorphic to the group  $\Gamma_\omega$.  The isomorphism is implemented by the map $\phi_\omega$ sending $\sigma_{(u_n.rv_n, i)}$ to $x_{(u_n.rv_n, i)}$. In particular, if $\sigma_{(w,i)}\in G_{T,\omega}'$, then $\phi_\omega(\sigma_{(w,i)}) = x_{(w,i)}$.
\end{lemma}
\begin{proof}

 Note that the relations (\ref{EqPermRel1}), (\ref{EqPermRel2}), (\ref{EqPermRel4}) are analogous to the relations (\ref{EqRel1}), (\ref{EqRel2}), and (\ref{EqRel3}). The cylinder sets appearing in the relations (\ref{EqPermRel3}) and (\ref{EqPermRel5}) are 3-disjoint (Remark \ref{Remark3disjoint}). Hence, these relations are consequences of the relation (\ref{EqRel5}). The relation (\ref{EqPermRel6}) follows from  (\ref{EqRel4}). Thus, the relations defining the group $G_{T,\omega}'$ (see Lemma \ref{LemmaGTW_presentation}) are a subset of the relations in the group $\Gamma_\omega$.  It follows that  the mapping  $\phi_\omega : G_{T,\omega}' \rightarrow \Gamma_\omega$ given by  $\phi_\omega(\sigma_{(u_n.rv_n, i)}) = x_{(u_n.rv_n, i)}$ gives rise to a group homomorphism.

 Since $\Gamma_\omega \neq \{1\}$ and the group $G_{T,\omega}'$ is simple (Proposition \ref{PropositionGroupPermutations}), we obtain that $\phi_\omega$ is a group isomorphism. This completes the proof of the lemma.
\end{proof}

Fix  $\omega,\omega'\in \Omega$ lying in distinct $T$-orbits. By construction  of $\phi_\omega$ (Lemma \ref{LemmaPresentationPermutationGroup}), if $r\in G_{T,\omega}'\cap G_{T,\omega'}'$, then $\phi_\omega(r) = \phi_{\omega'}(r)$.

 Since $G_T' = G_{T,\omega}\cdot G_{T,\omega'}'$  (Theorem \ref{TheoremPermutationProduct}), we can extend the map $\phi_\omega$ from Lemma \ref{LemmaPresentationPermutationGroup} to $\phi : G_T' \rightarrow \Gamma$ as follows.  For each $g\in G_{T}'$,  find elements $g_\omega\in G_{T,\omega}'$ and $g_{\omega'}\in G_{T,\omega'}'$ such that $g = g_\omega g_{\omega'}$. Set $\phi(g) = \phi_\omega(g_\omega) \phi_{\omega'}(g_{\omega'})$. The definition is independent of the choice of $g_\omega$ and $g_{\omega'}$. Indeed, if $g = h_\omega h_{\omega'}$ is another factorization with $h_\omega\in G_{T,\omega}'$ and $h_{\omega'}\in G_{T,\omega'}'$, then we can find $r\in G_{T,\omega}' \cap G_{T,\omega'}'$ such that $h_{\omega} = g_\omega r$ and $h_{\omega'} = r^{-1}g_{\omega'}$. Since the map $\phi$ restricted to $G_\omega$ ($G_{\omega'}$) equals the homomorphism $\phi_\omega$ ($\phi_{\omega'}$), we obtain that $$\phi(g) = \phi_\omega(h_\omega) \phi_{\omega'}(h_{\omega'}) = \phi_\omega(g_\omega r) \phi_{\omega'}(r^{-1}g_{\omega'}) = \phi_\omega(g_\omega)\phi_\omega(g_{\omega'}).$$

\begin{lemma} (1) The mapping $\phi: G_T'\rightarrow \Gamma$ is an epimorphism.

(2) $\phi(\sigma_{(w,k)}) = x_{(w,k)}$ for every $w\in L(\Omega)$ and $k\in \mathbb Z$, i.e., $\phi = \psi^{-1}$. \end{lemma}
\begin{proof}  Let $Q  = \prod_{m=1}^k \sigma_{(w_m,d_m)}$, where $(w_m,d_m)\in L(\Omega)\times \mathbb Z$, $m=1,\ldots,k$. To prove that $\phi$ is an onto homomorphism, it suffices to show that $\phi(Q) = \prod_{m=1}^kx_{(w_m,d_m)}$.

   Fix $\omega,\omega'\in \Omega$ lying in distinct $T$-orbits. The idea of the proof is the following. We  will explicitly  show how to  rewrite $Q$ in the form $Q = Q_\omega Q_{\omega'}$, where $Q_\omega\in G_{T,\omega}'$ and $Q_{\omega'}\in G_{T,\omega'}'$, whereby getting that $\phi(Q) = \phi(Q_\omega)\phi(Q_{\omega'})$.  The rewriting process will only rely on the relations present in Corollary \ref{CorollarySigmaRelations}, i.e., the relations that have their direct counterparts in the group $\Gamma$. This will imply that repeating the same rewriting process  for $x = \prod_{i=1}^kx_{(w_i,d_i)}$ we will obtain that $x = \phi(Q_\omega)\phi(Q_{\omega'})$.

  Using Proposition \ref{PropositionExistKRPartitions} choose $n$ so big that for $u = \omega[-n,-1]$ and $v=\omega[0,n]$ (i)  the sets $(u.v,i):=T^i[u.v]$, $-3k\leq i\leq 3k+2 $, are disjoint; (ii) the Kakutani-Rokhlin partition $\Xi$ associated with the family of return words $ R_{u.v}$ refines the cylinder sets $(w_m,d_m)$, $m=1,\ldots, k$; (iii) for each $m=1,\ldots,k$ and $-3k\leq i \leq 3k$, the cylinder set $(u.v,i)$ is either contained in $(w_m,d_m)$ or disjoint from it.  Since $\omega$ and $\omega
 '$ lie in distinct $T$-orbits, we can further ensure that  $n$ is chosen so big that  $\omega' \notin B:=\bigcup_{i=-3k}^{3k+2}T^i[u.v]$. Note that the union is disjoint.

 Since  the partition $\Xi_n$ refines the cylinder set $(w_m,d_m)$,  $m=1,\ldots,k$, the element $\sigma_{(w_m,d_m)}$ can be factored out as
\begin{equation}\label{EqDecomposition1} \sigma_{(w_m,d_m)} = \prod_{r\in R_{u.v}}\prod_{\substack{ 0\leq i\leq |r|-1\\
(u.rv,i)\subset (w_m,d_m)} } \sigma_{(u.rv,i)}.\end{equation}

Note that  the cylinder sets $(u.rv,i)$ appearing in Equation (\ref{EqDecomposition1})  are 3-disjoint as  disjoint subsets of $(w_m,d_m)$ (Remark \ref{Remark3disjoint}).  Thus,  the elements $\sigma_{(u.rv,i)}$ in (\ref{EqDecomposition1})  commute.

For each $m=1,\ldots,k$, set \begin{equation}\label{EqP_m}P_m =  \prod_{r\in R_{u.v}}\prod_{\substack{ 3m \leq i\leq |r|-1-3m\\
 (u.rv,i) \subset (w_m,d_m)}} \sigma_{(u.rv,i)}\end{equation}
and

\begin{equation}\label{EqQ_m}Q_m =  \prod_{r\in R_{u.v}}\prod_{\substack{ 0 \leq i < 3m \\
 |r|-1-3m < i \leq |r| -1 \\
 (u.rv,i) \subset (w_m,d_m)}} \sigma_{(u.rv,i)},\end{equation}
 It follows that
 \begin{equation}\label{EqSigmaDecomposition}\sigma_{(w_m,d_m)} = P_mQ_m.\end{equation}

  Observe that $Q_m$ and $P_l$, $l>m$, have disjoint supports, and, thus, they commute.  It follows that
  \begin{equation}\label{EqQDecomposition}Q = P_1Q_1\cdots P_k Q_k = (P_1\cdots P_k) (Q_1\cdots Q_k).\end{equation}
 By Proposition \ref{PropositionGroupPermutations}, we have that $P_m\in G_{T,\omega}'$, $m=1,\ldots,k$.

 Now we will  show that $Q_m \in G_{T,\omega'}'$ for each $m=1,\ldots,k$. Since for each return word $r\in R_{u.v}$, $v$ is a prefix of $rv$ and $u$ is a suffix of $ur$, we obtain that
\begin{equation*}  \sigma_{(u.v,i)} = \prod_{r\in R_{u.v}}\sigma_{(u.rv,i)}.\end{equation*}
Therefore, for each $m=1,\ldots,k$,
\begin{equation*} Q_m =  \prod_{\substack{ - 3m \leq i < 3m \\
 (u.v,i)\subset (w_m,d_m)}} \sigma_{(u.v,i)}.\end{equation*}
  Since each element $\sigma_{(u.v,i)}$, $i=-3k,\ldots,3k$, is supported by $B$, we conclude that $\sigma_{(u.v,i)}$ preserves the forward orbit of $\omega'$ and, in view of Proposition \ref{PropositionGroupPermutations}, $\sigma_{(u.v,i)}\in G_{T,\omega'}'$. Therefore, $Q_m\in G_{T,\omega'}'$.

 Consider $x_{(w_m,d_m)}$, $m=1,\ldots,k$. Define elements $p_m,q_m\in \Gamma$ as in Equations (\ref{EqP_m}) and (\ref{EqQ_m}), but using $x_{(u.rv,i)}$ for $\sigma_{(u.rv,i)}$. Note that  Equation (\ref{EqSigmaDecomposition}) relied exclusively on relations from Corollary \ref{CorollarySigmaRelations}. Hence, we can repeat the same argument to check that   $$x_{(w_m,d_m)} = p_mq_m.$$ Since the mapping $\phi$ restricted to $G_{T,\omega}'$ ($G_{T,\omega'}'$)  is an isomorphism, we get that  $p_m = \phi(P_m)$ and $q_m = \phi(Q_m)$. It follows that $$x_{(w_m,d_m)}= p_mq_m = \phi(P_m)\phi(Q_m) = \phi(\sigma_{(w_m,d_m)}).$$ In particular, this equation establishes the second statement of the lemma.    Note  that in view of the relation (\ref{EqRel5}), the elements $q_m$ and $p_l$, $l>m$,  commute.  Hence, by Equation (\ref{EqQDecomposition}), we get that $$\phi(Q) = (p_1\cdots p_k)(q_1\cdots q_k) = \prod_{m=1}^kx_{(w_m,d_m)}.$$
 This completes the proof of  the lemma.
 \end{proof}


 Since the group  $G_T'$ is simple, we obtain that $\phi : G_T'\rightarrow \Gamma$ is an isomorphism.   Theorem \ref{TheoremGenerators} says that  the group $G_T'$ is generated by $\{\sigma_{(abc,1)}\}$, $abc\in L_3(\Omega)$.
 As $\phi(\sigma_{(abc,1)}) = x_{(abc,1)}$, we conclude that the group $\Gamma$ is generated by $\{x_{(abc,1)}\}$, $abc\in L_3(\Omega)$.  This completes the proof of the theorem.
\end{proof}


The following version of Theorem \ref{TheoremMain} presents a set of Tietze transformations needed to switch between the two sets of generators. Let $F(S)$ be the free group with basis $$S = <x_{(w,1)} : w\in L_3(\Omega)>.$$ If $w\notin L_3(\Omega)$, we assume that $x_{(w,1)} = \varepsilon$, the empty word. For $a,b,c\in A$, let
\begin{equation}\label{Equation_X_abc}x_{[a.]} = \prod_{b,c\in A}x_{(abc,1)}\mbox{ and }x_{[.bc]} = \prod_{a\in A}x_{(abc,1)}\end{equation}
and  for each $(w,k)$ with $w=w_0\cdots w_{n-1}\in L_n(\Omega)$, $n\geq 4$, and  $2\leq k \leq n-1$, let
\begin{equation}\label{EquationFreeConvolution1}x_{(w,1)} =  x_{[.w_0w_1]} \ast (x_{[.w_1w_2]}\ast \cdots \ast(x_{[.w_{n-4}w_{n-3}])} \ast x_{(w_{n-3}w_{n-2}w_{n-1},1)})\cdots)\end{equation}
and
\begin{equation}\label{EquationFreeConvolution2}x_{(w,k)}  =   (\cdots ((x_{(w,1)}
 \ast  x_{[w_2.]}) \ast x_{[w_3.]}) \ast \cdots \ast x_{[w_{k}.]}).\end{equation}

\begin{theorem}\label{TheoremTietzeTransformations} Let $(\Omega,T)$ be a minimal subshift satisfying the condition $(\dag)$.  The group $G_T'$ has the presentation $<S|R>$, where
$R$ is the following set of relations: for every $w,v\in L(\Omega)$,  $1\leq i \leq |w|-1$, $1\leq j \leq |v|-1$,  and a  partition $C$ of $(w,i)$ into cylinder sets of the form $(s,k)$, $1\leq k \leq |s|-1$,
\begin{flalign}
& \left(x_{(w,i)}\right)^3=1   \label{EqVerRel1} \\
& \left(x_{(w,i)}\cdot x_{(w,i+1)}\right)^2 = 1  \label{EqVerRel2} \\
& x_{(w,i+1)} = x_{(w,i)} \ast x_{(w,i+2)}   \label{EqVerRel3} \\
& x_{(w,i)} = \prod_{(s,k) \in C}x_{(s,k)}  \label{EqVerRel4}
\end{flalign}
and
\begin{equation}  \label{EqVerRel5}  [x_{(w,i)},x_{(v,j)}] = 1,\mbox{ whenever }(w,i),(v,j)\mbox{ are 3-disjoint}.
\end{equation}
\end{theorem}
\begin{proof} Denote by $\Gamma$ the quotient group $F(S)/N$, where $N$ is the normal closure of $R$ in $F(S)$. Consider an arbitrary cylinder set $(w,i)$, $w\in L(\Omega)$, $i\in \mathbb Z$. Observe that cylinder sets in Equation (\ref{EqVerRel4}) are 3-disjoint (Remark \ref{Remark3disjoint}). Hence the corresponding factors commute (see the relation (\ref{EqVerRel5})).   Find a   partition $C$ of $(w,i)$ into cylinder sets of the form $(s,k)$, $1\leq k \leq |s|-1$. Define $$ x_{(w,i)} = \prod_{(s,k) \in C}x_{(s,k)}.$$ We claim that the definition of $x_{(w,k)}$ does not depend on a particular choice of the partition. Indeed, if $C'$ is another partition of $(w,i)$, then we can take a cylinder partition $Q$ that refines both $C$ and $C'$. Then the relation (\ref{EqVerRel4}) implies that $ x_{(w,i)} = \prod_{(s,k) \in Q}x_{(s,k)}.$ Furthermore, this shows that the relation (\ref{EqVerRel4}) holds for arbitrary clopen sets $(w,i)$ and arbitrary cylinder partition $C$.

Let $(w,i)$ and $(v,j)$ be arbitrary 3-disjoint cylinder sets. Consider clopen partitions $C_w$ and $C_v$ of $(w,i)$ and $(v,j)$, respectively. Assume that each cylinder atom $(s,k)$ of $C_w$ (of $C_v$) is of the form $1\leq k \leq |w| -1 $  ($1\leq k \leq |v| -1 $). Note that atoms of $C_w$ and $C_v$ are 3-disjoint. It follows from the relation (\ref{EqVerRel5}) and the commutator identities that $$[x_{(w,i)},x_{(v,j)}] = [\prod_{(s,k)\in C_w} x_{(s,k)},\prod_{(q,r)\in C_w} x_{(q,r)}] = 1.$$ This shows that the relation (\ref{EqVerRel5}) holds for arbitrary 3-disjoint cylinder sets.

Using the relations (\ref{EqVerRel4}),(\ref{EqVerRel5}), and the commutator identities, one can also check that the scope of the relations (\ref{EqVerRel1}) -- (\ref{EqVerRel3}) extends to arbitrary cylinder sets. We leave the details to the reader. Thus, the set of relations that define $\Gamma$ include those defining the group  $\Gamma_\Omega$ in Theorem \ref{TheoremMain}.  On the other hand, since $\Gamma_\Omega$ is isomorphic to $G_T'$, it follows from Proposition \ref{PropositionGroupWords} that the relations (\ref{Equation_X_abc}) -- (\ref{EquationFreeConvolution2}) also hold in $\Gamma_\Omega$.

Denote the generators of $\Gamma_\Omega$ by $y_{(w,i)}$. Since  $\Gamma$ and $\Gamma_\Omega$ are defined by the same set of relations, the map $\psi: \Gamma \rightarrow \Gamma_\Omega$ given by $\psi(x_{(w,i)} = y_{(w,i)})$ implements the group isomorphism.
\end{proof}

The following result describes TFGs with solvable word problem. We mention that an alternative proof of the ``only if''  part  of the following theorem has been earlier established by the authors  \cite[Theorem 2.15]{GrigorchukMedynets}.

\begin{theorem}\label{TheoremWordProblem}   The language of the system $(\Omega,T)$ is recursive if and only if the group $G_T'$ has solvable word problem.
\end{theorem}
\begin{proof}   In view of Proposition \ref{PropositionConjugateSubshift}, switching to a topologically conjugate subshift, if needed, we can assume that none of the word of length five has repeated letters.   We note that such a conjugation will not affect the recursiveness of the language.

Suppose that the language $L(\Omega)$ of $(\Omega,T)$ is recursive.  Theorem \ref{TheoremMain} and Proposition \ref{PropositionDecidabilityDisjointness} imply  that the set of defining relations is  recursively enumerable. Since the group $G_T'$ is simple, by Kuznetsov's theorem \cite[Proposition 13 on p. 259]{Cohen:1989} we conclude that  $G_T'$ has solvable word problem.

  Now assume that the group $G_T'$ has solvable word problem.    Consider the generators for $G_T'$ from Theorem \ref{TheoremGenerators}.   Fix an algorithm $\mathcal M$ that solves the word problem for $G_T'$.  For a word $w\in A^*$, $|w|>3$, consider the group element $\sigma_{(w,1)}$. In view of  Corollary \ref{CorollaryWordIdentity},  $\sigma_{(w,1)} \neq 1$ in $G_T'$  if and only if $w$ belongs to the language of $\Omega$.  To decide if a word of length two or one belongs to $L(\Omega)$, check if it is a subword of a word in $L_3(\Omega)$.  Then $w$ belongs to $L(\Omega)$ if and only if the algorithm $\mathcal M$ on input $\sigma_{(w,1)}$ has output ``No''.
\end{proof}


\end{document}